\documentclass[a4paper, 11pt]{amsart}
\usepackage[latin1]{inputenc}
\usepackage[ T1]{fontenc}
\usepackage[english]{babel}
\usepackage{amssymb}
\usepackage{amsmath}
\usepackage{amsthm}
\usepackage{amscd}
\usepackage{amsfonts}
\usepackage{stmaryrd}
\usepackage{pb-diagram}
\usepackage{epic,eepic,epsfig}
\usepackage{a4wide}
\usepackage{nextpage}
\usepackage{fancyhdr}
\usepackage{enumerate}

\newtheorem{prop}{Proposition}[section]
\newtheorem{corollary}[prop]{Corollary}

\newtheorem{lemma}[prop]{Lemma}

\newtheorem{theorem}[prop]{Theorem}

\newtheorem{proposition}[prop]{Proposition}

\theoremstyle{definition}
\newtheorem{rem}[prop]{Remark}
\newtheorem{ques}[prop]{Question}
\newtheorem{definition}[prop]{Definition}

\newcommand{\Card} {\mathop{\mathrm{Card}}}

\newcommand{\Reg} {\mathop{\mathrm{Reg}}}
\newcommand{\RegL} {\mathop{\mathrm{Reg}_L}}

\newcommand{\Vol} {\mathop{\mathrm{Vol}}}
\newcommand{\Covol} {\mathop{\mathrm{Covol}}}

\usepackage[OT2,T1]{fontenc}
\DeclareSymbolFont{cyrletters}{OT2}{wncyr}{m}{n}
\DeclareMathSymbol{\Sha}{\mathalpha}{cyrletters}{"58}

\begin{document}

\title{Regulators of elliptic curves}
\author{{P}ascal {A}utissier, {M}arc {H}indry and {F}abien {P}azuki}
\address{Institut de Math\'ematiques de Bordeaux, Universit\'e de Bordeaux,
351, cours de la Lib\'e\-ra\-tion, 
33405 Talence, France.}
\email{pascal.autissier@math.u-bordeaux.fr}
\address{Institut de Math\'ematiques de Jussieu-PRG,
UFR math\'ematiques de l'Universit\'e Paris 7 Denis Diderot,
B\^atiment Sophie Germain,
5 rue Thomas Mann,
F-75205 Paris CEDEX 13.}
\email{marc.hindry@imj-prg.fr}
\address{Department of Mathematical Sciences, University of Copenhagen,
Universitetsparken 5, 
2100 Copenhagen \O, Denmark.}
\email{fpazuki@math.ku.dk}

\thanks{The three authors are supported by ANR-17-CE40-0012 Flair. The first and third authors are supported by ANR-14-CE25-0015 Gardio. The third author is supported by the DNRF Niels Bohr Professorship of Lars Hesselholt. The authors would like to thank Ga\"el R\'emond for his remarks, and the referee for his work and his helpful comments.}
\maketitle

\vspace{0.5cm}

\noindent \textbf{Abstract.}
We study the regulator of the Mordell-Weil group of elliptic curves over number fields, functions fields of characteristic zero or function fields of characteristic $p>0$. We prove a new Northcott property for the regulator of elliptic curves of rank at least $4$ defined over a number field. In the function field case, we obtain that non-trivial families of elliptic curves with bounded regulators and bounded ranks are limited families. We conclude by asking new questions.

{\flushleft
\textbf{Keywords:} Heights, abelian varieties, regulators, Mordell-Weil.\\
\textbf{Mathematics Subject Classification:} 11G50, 14G40. }

\begin{center}
---------
\end{center}

\begin{center}
\textbf{R\'egulateurs des courbes elliptiques}.
\end{center}

\noindent \textbf{R\'esum\'e.}
On \'etudie le r\'egulateur du groupe de Mordell-Weil des courbes elliptiques d\'efinies sur un corps de nombres, sur un corps de fonctions en caract\'eristique z\'ero ou sur un corps de fonctions en caract\'eristique $p>0$. On montre une propri\'et\'e de Northcott pour le r\'egulateur de familles de courbes elliptiques de rang au moins $4$ d\'efinies sur un corps de nombres. Dans le cas des corps de fonctions, on montre qu'une famille non triviale de courbes elliptiques de r\'egulateur et rang born\'es est une famille limit\'ee. On propose de nouvelles questions en derni\`ere partie.

{\flushleft
\textbf{Mots-Clefs:} Hauteurs, vari\'et\'es ab\'eliennes, r\'egulateurs, Mordell-Weil.\\
}

\begin{center}
---------
\end{center}

\thispagestyle{empty}


\maketitle

\section{Introduction}

Let $K$ be a number field or a function field of transcendence degree one over its field of constants $k$. Let $A$ be an abelian variety over $K$, where we assume $\mathrm{Tr}_{K/k}(A)=0$ in the case of function fields. Then the group of $K$-rational points $A(K)$ is finitely generated by N\'eron's thesis, generalizing the Mordell-Weil Theorem. For a modern exposition of the Lang-N\'eron results, we refer to \cite{Co06}. Giving more arithmetic information about the Mordell-Weil group is a difficult problem in general, as explained for instance in \cite{Hin}. A central notion attached to the group of $K$-rational points is the regulator. Its appearance in the strong form of the Birch and Swinnerton-Dyer conjecture, or in statements \textit{\`a la Brauer-Siegel} given in \cite{HiPa} stresses the importance of it when studying the arithmetic of elliptic curves and of abelian varieties in general. We refer to \cite{Paz14, Paz16} for the proof of a Northcott property of regulators of number fields, and for a conjectural similar statement for abelian varieties. The main result obtained here concerns the regulator of elliptic curves over number fields. We use $h(.)$ for the Weil height of algebraic numbers.

\begin{theorem}\label{blichreg}
Let $E$ be an elliptic curve defined over the number field $K$. Let $d=[K:\mathbb{Q}]$ and $m$ be the Mordell-Weil rank of $E(K)$. Assume $m\geq 1$. Denote $h=\max\{1, h(j_E)\}$. Then there is a $c(d,m)>0$ depending at most on $d$ and on $m$, such that $$\frac{\Reg(E/K)}{(\Card E(K)_{tors})^2}\geq c(d,m)\, h^{\frac{m-4}{3}}\Big(\log (3h)\Big)^{\frac{2m+2}{3}},$$
and one may take $c(d,m)=\frac{c^m}{m^m d^{m+2}(\log(3d))^2}$, where $c>0$ is an absolute constant.
\end{theorem}

It implies immediately the following corollary.

\begin{corollary}\label{NorthReg}
Let $K$ be a number field and $m\geq4$ be an integer. There are at most finitely many $\overline{K}$-isomorphism classes of elliptic curves over $K$ with rank $m$ and bounded regulator.
\end{corollary}

Theorem \ref{blichreg} has other interesting consequences. For instance, combining with classical estimates on the asymptotic of the number of points of bounded height on elliptic curves (see for instance \cite{Lang} Theorem 7.4 page 126 and Theorem 7.5 page 127), one remarks the following.

\begin{rem}\label{Count}
Pick an elliptic curve $E$ of rank $m\geq 1$ over the number field $K$ of degree $d$ and let $\hat{h}_E$ be the N\'eron-Tate height on $E$. Then
$$\Card \{P\in{E(K)}\,\vert\, \hat{h}_E(P)\leq T\}\sim c_E\, T^{m/2},$$
where $T\to +\infty$ is a real number, and 
$$c_E=\frac{\pi^{m/2}}{\Gamma(\frac{m}{2}+1)} \frac{\Card E(K)_{tors}}{\sqrt{\Reg(E/K)}},$$
$\Gamma$ being the gamma function. One sees here that $c_E$ is essentially the inverse of the square root of the quotient studied in Theorem \ref{blichreg}, which implies the estimate 
\begin{equation}
c_E\leq c'(d,m)\, h^{-\frac{m-4}{6}}\Big(\log (3h)\Big)^{-\frac{m+1}{3}}
\end{equation}
where $c'(d,m)$ is a positive real number depending at most on $d$ and $m$.
\end{rem}

The proof of Theorem \ref{blichreg} relies on a general approach to bound the regulator from below using counting arguments of points of small height and explained in Lemmas \ref{success} and \ref{successbis}. Direct application of counting results of David \cite{Da97} would then give $16$ instead of $4$ in Corollary \ref{NorthReg}. To be able to go down to $4$, we combine the counting results of David \cite{Da97} and of Petsche \cite{Pet06}, the Lemma \ref{success2} and provide a non-trivial optimization of parameters.

We also obtain similar statements in the function field case, with the natural restrictions inherent to this other context (non-isotriviality, bounded inseparability degree, limited families).

The article starts by recalling the relevant definitions in Section \ref{section def}. In Section \ref{general} we give the general approach to obtain lower bounds on the regulator of an abelian variety in any dimension. In Section \ref{ellNF} we focus on elliptic curves defined over a number field and we prove Theorem \ref{blichreg}. In Sections \ref{ellFF0} and \ref{ellFFp} we give the counterparts respectively for elliptic curves over function fields in characteristic zero and characteristic $p>0$. In Section \ref{questions} we propose new questions about the growth of regulators of elliptic curves and abelian varieties.

Throughout the text, we use the logarithm $\log$ normalized by $\log(e)=1$.

\section{Definitions}\label{section def}

\subsection{Number fields} Let $K$ be a number field of degree $d$. Let $M_K$ be a complete set of inequivalent absolute values $\vert.\vert_v$ on $K$, given local degrees $n_v$ and normalized by $\vert p\vert_v=1/p$ for any $v\vert p$, for the underlying rational prime $p$. Hence the product formula holds: $$\forall x\in{K-\{0\}},\quad \sum_{v\in{M_K}}n_v\log \vert x\vert_v =0.$$

The height on $K$ is defined by the formula $$h(x)=\frac{1}{d}\sum_{v\in{M_{K}}}n_v \log\max\{1, \vert x\vert_v\}.$$

Let $E$ be an elliptic curve defined over the number field $K$, with neutral element $O$. We define the N\'eron-Tate height on the group of rational points $E(K)$ with respect to $(O)$ by $$\widehat{h}_E(P)=\frac{1}{2}\lim_{n\to\infty}\frac{1}{n^2}h(x([n]P)).$$

Let $I\subset\mathcal{O}_K$ be a non-zero ideal of the ring of integers of $K$. Then one also defines the height of $I$ by $h(I)=\frac{1}{d}\log N_{K/\mathbb{Q}}(I)$. For any elliptic curve with conductor $F(E/K)$ and minimal discriminant $\Delta(E/K)$, we can now define the Szpiro quotient by $$\sigma_{E/K}=\frac{h(\Delta(E/K))}{h(F(E/K))},$$ with the convention that in the case of good reduction everywhere, we fix $\sigma_{E/K}=1$.

\subsection{Function fields} Let $K=k(\mathcal{C})$ be the function field of a smooth projective and geometrically connected curve $\mathcal{C}$ defined over $k$ and of genus $g$. We denote by $M_K$ a complete set of inequivalent valuations $v(.)$ and given with local degrees $n_v$ (from each point $a\in{\mathcal{C}}$ one has a valuation given for $x\in{K}$ by $a(x)=\mathrm{ord}_a(x)$, and $n_v=[k(a):k]$), which gives a normalization such that for any non-zero element $x\in{K}$, the product formula holds:  $$ \sum_{v\in{M_K}}n_v v(x)=0.$$ 

The height on $K$ is then defined as $h(0)=0$ and for any $x\in{K-\{0\}}$, we pose $$h(x)=\sum_{v\in{M_K}}n_v \max\{0, -v(x)\}.$$

Let $E$ be an elliptic curve defined over the function field $K$. We define the N\'eron-Tate height on the group of rational points $E(K)$ with respect to $(O)$ by $$\widehat{h}_E(P)=\frac{1}{2}\lim_{n\to\infty}\frac{1}{n^2}h(x([n]P)).$$

A divisor $I$ on $K$ is a formal sum $\sum_{v\in{M_K}}a_v \cdot v$ where $a_v\in{\mathbb{Z}}$ is zero for all but finitely many places $v$. The divisor is effective if for all $v$, $a_v\geq0$. In that case we pose $$h(I)=\sum_{v\in{M_K}}n_v a_v =\deg(I).$$

\subsection{Regulators of abelian varieties}

Let $K$ be a number field or a function field of transcendence degree one over its field of constants $k$.
Let $A/K$ be an abelian variety over the field $K$ polarized by an ample and symmetric line bundle $L$. We assume that $A$ has trace zero if $K$ is a function field. Let $m_K$ be the Mordell-Weil rank of $A(K)$. Let $\widehat{h}_{A,L}$ be the N\'eron-Tate height associated with the pair $(A,L)$. Let $<.,.>_L$ be the associated bilinear form, given by $$<P,Q>_L=\frac{1}{2}\Big(\widehat{h}_{A,L}(P+Q)-\widehat{h}_{A,L}(P)-\widehat{h}_{A,L}(Q)\Big)$$ for any $P,Q\in{A(K)}$. This pairing on $A\times A$ depends on the choice of $L$.

\begin{definition}\label{reg abvar}
Let $P_1, ..., P_{m_K}$ be a basis of the lattice $A(K)/A(K)_{\mathrm{tors}}$, where $A(K)$ is the Mordell-Weil group. The regulator of $A/K$ is defined by $$\mathrm{Reg}_{L}(A/K)= \det(<P_i,P_j>_{L\; 1\leq i,j\leq m_K}),$$ where by convention an empty determinant is equal to $1$.
\end{definition}

As for the height, the regulator from Definition \ref{reg abvar} depends on the choice of an ample and symmetric line bundle $L$ on $A$. 

One usually defines the regulator in a more intrinsic way, that doesn't depend on the choice of $L$. Start with the natural pairing on the product of $A$ with its dual abelian variety $\check{A}$ given by the Poincar\'e line bundle $\mathcal{P}$: for any $P\in{A(K)}$ and any $Q\in{\check{A}(K)}$, define $<P,Q>=\widehat{h}_{A\times \check{A},\mathcal{P}}(P,Q)$. Next choose a base $P_1, ..., P_{m_K}$ of $A(K)$ modulo torsion and a base $Q_1, ..., Q_{m_K}$ of $\check{A}(K)$ modulo torsion. Then define $$\Reg(A/K)= \vert\det(<P_i, Q_j>_{1\leq i,j\leq m_K})\vert.$$

Let us recall how these two regulators are linked (see for instance \cite{Hin}). Let $\Phi_L:A\to \check{A}$ be the isogeny given by $\Phi_L(P)=t_{P}^{*}L\otimes L^{-1}$. By Proposition 9.3.6 page 291 of \cite{BG}, one has the formula $$<P,Q>_L=\frac{1}{2}<P,\Phi_L(Q)>.$$ 
Hence if $u$ is the index of the subgroup $\Phi_L(\mathbb{Z}P_1\oplus ... \oplus\mathbb{Z}P_{m_K})$ in $\check{A}(K)$ modulo torsion, one has 

\begin{equation}\label{regulators}
\mathrm{Reg}_L(A/K)=u2^{-m_K}\mathrm{Reg}(A/K).
\end{equation}

Let us remark that when $L$ induces a principal polarization, the index $u$ is equal to $1$. Hence in the case of elliptic curves, we have $L=(O)$ and $$\mathrm{Reg}_{L}(E/K)=2^{-m_K}\mathrm{Reg}(E/K).$$

\section{Lower bounds for regulators of abelian varieties}\label{general}

\subsection{Minkowski and regulators}

A natural idea is to apply Minkowski's successive minima inequality to the Mordell-Weil lattice, we give it here as a lemma.

\begin{lemma}\label{success}
Let $K$ be a number field or a function field of transcendence degree one over its field of constants $k$. Let $A$ be an abelian variety over the field $K$, let $L$ be an ample and symmetric line bundle on $A$. We assume that $A$ has trace zero if $K$ is a function field. Let $m$ be the Mordell-Weil rank of $A(K)$. Assume $m\geq 1$. Let $\Lambda=A(K)/A(K)_{\mathrm{tors}}$. Then for any $i\in\{1, ..., m\}$, let us denote the Minkowski $i$th-minimum of $(\Lambda, \sqrt{\widehat{h}_{A,L}})$ by $\lambda_i$. We have
\begin{equation}\label{Minkowski2}
 \lambda_1 \cdots \lambda_{m}\leq m^{{m/2}}  (\mathrm{Reg}_{L}(A/K))^{1/2}.
\end{equation}
\end{lemma}

\begin{proof}

Minkowski's successive minima inequality reads

\[
 \lambda_1 \cdots \lambda_{m}\leq \frac{2^m \mathrm{vol}((\Lambda\otimes\mathbb{R})/\Lambda)}{\mathrm{vol}(B(0,1))}  .
\]

As the volume of a Euclidean unit ball in $\Lambda\otimes\mathbb{R}\simeq\mathbb{R}^m$ is $\frac{\pi^{m/2}}{\Gamma(\frac{m}{2}+1)}$, and as $\mathrm{vol}((\Lambda\otimes\mathbb{R})/\Lambda)=(\mathrm{Reg}_{L}(A/K))^{1/2}$ we obtain

\[
 \lambda_1 \cdots \lambda_{m}\leq \frac{2^m \Gamma(\frac{m}{2}+1)}{\pi^{m/2}}  (\mathrm{Reg}_{L}(A/K))^{1/2}.
\]

\noindent Now for any $m\geq1$, one gets easily $ 2^m \Gamma(\frac{m}{2}+1)\leq (\pi m)^{\frac{ m}{2}}$, it leads to the result.

\end{proof}

\begin{lemma}\label{successbis}
Let $A$ be an abelian variety over a number field $K$, let $L$ be an ample and symmetric line bundle on $A$. Let $m$ be the Mordell-Weil rank of $A(K)$. Let $\Lambda=A(K)/A(K)_{tors}$. Assume $m\geq 1$. Assume one has $$\Card\{P\in{\Lambda} \; \vert \; \widehat{h}_{A,L}(P) \leq H\} \leq C,$$ for $H, C$ two fixed positive real numbers. Then for any $1\leq i \leq m$, the $i$-th successive minimum of $\Lambda$ satisfies
\begin{equation}\label{lambis}
 \lambda_{i}^2 \geq  \frac{H}{i^2 C^{2/i}}.
\end{equation}
\end{lemma}

\begin{proof}
Let $P_1,\ldots, P_m$ be linearly independent points of $\Lambda$ satisfying $\lambda_{i}^2=\widehat{h}_{A,L}(P_i)$ for any $1\leq i\leq m$. Consider the set of all sums $a_1P_1+\cdots+a_iP_i$ where for any index $ j\in\{1, \ldots, i\}$ the integral coefficient $a_j$ satisfies $0\leq a_j\leq C^{1/i}$. It contains more than $C$ points, hence (by the assumption) at least one of them has height greater than $H$. This point can be written as $P=a_1P_1+\cdots+a_iP_i$, and $\widehat{h}_{A,L}(P)\leq i^2 \max\{\widehat{h}_{A,L}(a_j P_j)\vert 1\leq j\leq i\}$ by the triangular inequality for the norm $\sqrt{\widehat{h}_{A,L}(.)}$, hence $$H\leq \widehat{h}_{A,L}(P) \leq i^2 C^{2/i} \lambda_i^2.$$
\end{proof}

Lemma \ref{success} and Lemma \ref{successbis} combine to give the following.

\begin{corollary}\label{reg2}
Let $A$ be an abelian variety over a number field $K$, let $L$ be an ample and symmetric line bundle on $A$. Let $m$ be the Mordell-Weil rank of $A(K)$. Assume $m\geq 1$. Assume one has $$\Card\{P\in{A(K)/A(K)_{tors}} \; \vert \; \widehat{h}_{A,L}(P) \leq H\} \leq C,$$ for $H, C$ two fixed positive real numbers. Then
\begin{equation}\label{Minkow}
 \mathrm{Reg}_{L}(A/K) \geq  \frac{H^m}{m^m (m!)^2}\prod_{i=1}^{m}\frac{1}{C^{2/i}}.
\end{equation}
\end{corollary}

\subsection{Regulators and van der Corput}

We show how to improve on Corollary \ref{reg2} by invoking van der Corput.

\begin{lemma}\label{success2}
Let $A$ be an abelian variety over a number field $K$, let $L$ be an ample and symmetric line bundle on $A$. Let $m$ be the Mordell-Weil rank of $A(K)$. Assume $m\geq 1$. Assume one has $$\Card\{P\in{A(K)/A(K)_{tors}} \; \vert \; \widehat{h}_{A,L}(P) \leq H\} \leq C,$$ for $H, C$ two fixed positive real numbers. Then
\begin{equation}\label{Blich}
 \mathrm{Reg}_{L}(A/K) \geq  \frac{H^m}{m^{m} C^2}.
\end{equation}
\end{lemma}

\begin{proof}
Consider $\mathcal{K}=\{P\in{A(K)\otimes \mathbb{R}} \; \vert \; \widehat{h}_{A,L}(P) \leq H\}$. It is a compact symmetric convex set (hence with finite volume) in $A(K)\otimes \mathbb{R}\simeq \mathbb{R}^m$. We apply van der Corput's result\footnote{In the spirit of Blichfeldt's principle.}, see for instance Theorem 7.1 in \cite{GL}, to the lattice $\Lambda=A(K)/A(K)_{tors}$ to obtain $$\Card(\mathcal{K}\cap \Lambda)\geq \frac{\Vol(\mathcal{K})}{2^m\Covol(\Lambda)}= \frac{H^{m/2}}{2^m (\mathrm{Reg}_L(A/K))^{1/2}}\frac{\pi^{m/2}}{\Gamma(\frac{m}{2}+1)}.$$ This gives directly $\frac{H^{m/2}}{m^{m/2}(\mathrm{Reg}_L(A/K))^{1/2}}\leq C$, hence the claim.
\end{proof}

\begin{rem}\label{success3}
Lemma \ref{success2} implies directly (with the same notation) that if
$$\Card\{P\in{A(K)} \; \vert \; \widehat{h}_{A,L}(P) \leq H\} \leq C,$$ then one has 
\begin{equation}\label{Blich2}
 \mathrm{Reg}_{L}(A/K) \geq  \frac{H^m}{m^{m} C^2} (\Card A(K)_{tors})^2.
\end{equation}
\end{rem}

\section{Regulators of elliptic curves over number fields}\label{ellNF}

To warm up, we will first combine Lemma \ref{success} with various height lower bounds.

\subsection{Szpiro and finiteness}

Let us first extract Corollary 4.2 page 430 from \cite{HiSi3}. 

\begin{theorem}(Hindry-Silverman)\label{nf1}
Let $K$ be a number field of degree $d$. Let $E/K$ be an elliptic curve and let $P\in{E(K)}$ be a non-torsion point. Then $$\hat{h}_E(P)\geq (20\sigma_{E/K})^{-8d} 10^{-4\sigma_{E/K}} h(E),$$ where $h(E)=\frac{1}{12}\max\{h(\Delta(E/K)), h(j_E)\}$, with $\Delta(E/K)$ the minimal discriminant of $E/K$ and $j_E$ its  $j$-invariant.
\end{theorem}

Note that \cite{Da97} or Theorem 2 page 259 of \cite{Pet06} give an inverse polynomial dependence in the Szpiro quotient and in the degree of the number field.

We get the following corollary.

\begin{corollary}\label{Kclass}
Let $K$ be a number field of degree $d$. There are at most finitely many $K$-isomorphism classes of elliptic curves over $K$ with non-trivial bounded rank, bounded regulator and bounded Szpiro quotient.
\end{corollary}

\begin{proof}
Combine Theorem \ref{nf1} and (\ref{Minkowski2}) to get the inequality 
\begin{equation}\label{RegSzp}
\mathrm{Reg}_L(E/K)\geq \frac{1}{m^m} \Big( (20\sigma_{E/K})^{-8d} 10^{-4\sigma_{E/K}} h(E) \Big)^m.
\end{equation}
As soon as the rank $m$ is positive and bounded, and as long as $\sigma_{E/K}$ is bounded, a bounded regulator implies a bounded height. Finiteness of the set of $K$-isomorphism classes follows, see Corollary 2.5 of \cite{CoSi} page 259 for instance. Note that our height is equivalent to Faltings's height $h_F(E)$ used by Silverman in the aforementioned corollary, one easily gets $h_F(E)\ll h(E)$ from the explicit formula for $h_F(E)$ given in the same reference.
\end{proof}

\subsection{Big ranks and finiteness}

Another approach is based on Corollaire 1.6 page 111 of \cite{Da97}, which we recall here.

\begin{proposition} (David)\label{DavidJNT}
Let $K$ be a number field of degree $d$. Let $E$ be an elliptic curve defined over $K$ of rank $m\geq 1$. Let $\Lambda=E(K)/E(K)_{tors}$. Let $i$ be an integer such that $1\leq i\leq m$ and $\lambda_i$ the $i$-th successive minimum. Let $h=\max\{1, h(j_E)\}$. Then $$\lambda_i^2\geq c_{16}(i) h^{(i^2-4i-4)/(4i^2+4i)}d^{-(i+2)/i}\Big(1+\frac{\log(d)}{h}\Big)^{-2/i},$$ and in particular $$\lambda_5^2\geq c_{17} h^{1/120}d^{-7/5}\Big(1+\frac{\log(d)}{h}\Big)^{-2/5}.$$
\end{proposition}

We deduce the following unconditional statement.

\begin{corollary}
Let $K$ be a number field. There are at most finitely many $\overline{K}$-isomorphism classes of elliptic curves over $K$ of fixed rank $m\geq 16$ with bounded regulator.
\end{corollary}

\begin{proof}
By combining Proposition \ref{DavidJNT} with (\ref{Minkowski2}) one obtains 

$$\mathrm{Reg}_{L}(E/K)\geq \frac{1}{m^{m}}\prod_{i=1}^m c_{16}(i) h^{(i^2-4i-4)/(4i^2+4i)}d^{-(i+2)/i}\Big(1+\frac{\log(d)}{h}\Big)^{-2/i}.$$

A direct computation shows that $$\sum_{i=1}^{16}\frac{i^2-4i-4}{4i^2+4i}\geq 0.009\quad\quad\quad \mathrm{and}\quad\quad\quad \sum_{i=1}^{15}\frac{i^2-4i-4}{4i^2+4i}\leq -0.16.$$ Hence as soon as $m\geq 16$, if $m$ and $d$ are fixed, a bounded regulator implies a bounded $j$-invariant, which implies the finiteness of the set of $\overline{K}$-isomorphism classes for a fixed base field $K$. 
\end{proof}

We will now show how to improve on this corollary using counting results from \cite{Da97, Pet06} and our Lemma \ref{success2}. We start with a small technical lemma that will help later on.

\begin{lemma}\label{chebych}
Let $K$ be a number field of degree $d$ and let $I$ be a non-zero integral ideal in $\mathcal{O}_K$. Let $S$ be the number of prime ideals dividing $I$. There is an absolute constant $c_0>0$ such that $$\log N(I)\geq c_0 S \log(\frac{S}{d}+2),$$ where $N(I)$ stands for the norm of the ideal $I$. 
\end{lemma}
\begin{proof}
Above each rational prime number $p$, there are at most $d$ prime ideals in $\mathcal{O}_K$. Let us write $S=dq+r$ the Euclidean division of $S$ by $d$, where $0\leq r<d$. Then one has $$\log N(I)\geq d \log p_1+\ldots+d\log p_q+r \log p_{q+1},$$ where $p_i$ denotes the $i$-th rational prime number. An easy bound from below is $\log p_i \geq \log (i+1)$ for any integer $i\geq1$, hence $$\log N(I) \geq d\log (q+1)! +r \log (q+2)\geq c_0 dq\log(q+3)+c_0r\log(q+3)\geq c_0 S \log(\frac{S}{d}+2).$$
\end{proof}

We now give the proof of Theorem \ref{blichreg}.

\begin{proof} (of Theorem \ref{blichreg} and Corollary \ref{NorthReg})
By Theorem 1.2 of \cite{Da97} and Proposition 8 of \cite{Pet06}, one has two independent upper bounds, with $c_1, c_2, c_3, c_4$ positive absolute constants:

\begin{itemize}

\item[(a)] $
\Card\{P\in{E(K)\,\vert\, \hat{h}_E(P) \leq c_1\frac{h_v(E)}{d} }\}\leq c_2 \frac{d h}{h_v(E)}(1+\frac{\log d }{h}),$ where one defines $h_v(E)=\max\{n_v \log\vert j_E\vert_v, \rho_v\}$, with $\rho_v=0$ if $v$ is a finite place of multiplicative reduction and $\rho_v=\sqrt{3}/2$ if $v$ is archimedean.
\\

\item[(b)] $
\Card\Big\{P\in{E(K)\,\vert\, \hat{h}_E(P) \leq  c_3\frac{h(\Delta(E/K))}{\sigma^2 }}\Big\}\leq c_4 d\sigma^2 \log(c_4d\sigma^2)$, where $\sigma=\sigma_{E/K}$ is the Szpiro quotient.
\\
\end{itemize}

First case: if one considers an elliptic curve $E/K$ with $h(j_E)<1$, one obtains by $(a)$, as $h=1$ in this case, for the choice $v$ archimedean and Lemma \ref{success2} (and Remark \ref{success3}):

\begin{equation}
\frac{\Reg(E/K)}{(\Card E(K)_{tors})^2}\geq \frac{\RegL(E/K)}{(\Card E(K)_{tors})^2}\geq \frac{c_1^m h_v(E)^{m+2}}{c_2^2d^{m+2}} \frac{1}{m^{m} (1+\log d)^{2}}\geq \frac{c_1^m \sqrt{3}^{m+2} }{c_2^2 (2d)^{m+2}} \frac{1}{m^{m} (1+\log d)^{2}},
\end{equation}
which is the claimed lower bound in this case.

Thus we may assume from now on that $h(j_E)\geq 1$. Let us denote $K'=K(E[12])$. The elliptic curve $E$ is semi-stable over $K'$ and the degree $d'=[K':\mathbb{Q}]\leq [K:\mathbb{Q}] \Card\mathrm{GL}_2(\mathbb{Z}/12\mathbb{Z}) = 4608 d$. Note that for any place $v$ of $K'$ one has  $$\Card\Big\{P\in{E(K)} \; \vert \; \widehat{h}_{E}(P) \leq c_1\frac{h_v(E)}{d'} \Big\} \leq \Card\Big\{P\in{E(K')} \; \vert \; \widehat{h}_{E}(P) \leq c_1\frac{h_v(E)}{d'} \Big\}.$$ Similarly, with $\sigma=\sigma_{E/K'}$ and $\Delta=\Delta(E/K')$, one has $$\Card\Big\{P\in{E(K)} \; \vert \; \widehat{h}_{E}(P) \leq c_3\frac{h(\Delta)}{\sigma^2} \Big\} \leq \Card\Big\{P\in{E(K')} \; \vert \; \widehat{h}_{E}(P) \leq c_3 \frac{h(\Delta)}{\sigma^{2}} \Big\}.$$ Both $h$ and $\widehat{h}_{E}(P)$ are invariant by field extension, so $(a), (b)$ and Lemma \ref{success2} (or Remark \ref{success3}, with $H=c_1h_v(E)/d'$ and $C=c_2d'h(1+(\log d')/h)/h_v(E)$ for the first inequality and $H=c_3 h(\Delta)/\sigma^2$ and $C=c_4 d'\sigma^2 \log(c_4d'\sigma^2)$ for the second inequality) imply that for any elliptic curve $E/K$ with $h(j_E)\geq 1$, and keeping $m$ the rank of $E$ over $K$, for any place $v$ of $K'$ of multiplicative reduction or archimedean,

\begin{equation}
\frac{\Reg(E/K)}{(\Card E(K)_{tors})^2}\geq \max\left\{\frac{c_1^m h_v(E)^{m+2}}{c_2^2 d'^{m+2} h^2 (1+\frac{\log d'}{h})^{2}}, \frac{c_3^m h(\Delta)^m}{\sigma^{2m} (c_4 d'\sigma^2 \log(c_4d'\sigma^2))^2}  \right\} \frac{1}{m^{m}}.
\end{equation}

\noindent This implies

\begin{equation}
\frac{\Reg(E/K)}{(\Card E(K)_{tors})^2}\geq \max\left\{\frac{c_1^m}{c_2^2 d'^{m+2}}\frac{ h_v(E)^{m+2}}{ h^2}, \frac{c_3^m}{ (c_4 d')^2} \frac{ h(\Delta)^{m}}{\sigma^{2m+4} (\log(c_4 d' \sigma^2))^2}  \right\} \frac{1}{m^{m} (1+\frac{\log d'}{h})^{2}}.
\end{equation}

Now using $d'\leq 4608 d$ and $\log(c_4d'\sigma^2)\ll\log(3h(\Delta))$ with an implied constant depending at most on $d$, and let $c(d,m)>0$ stand for a quantity depending at most on $d$ and $m$ and that may slightly change along the way, we get

\begin{equation}\label{pivot}
\frac{\Reg(E/K)}{(\Card E(K)_{tors})^2}\geq \max\left\{\frac{ h_v(E)^{m+2}}{ h^2},  \frac{ h(\Delta)^{m}}{\sigma^{2m+4} (\log(3h(\Delta)))^2}  \right\} c(d,m).
\end{equation}

Second case: assume $h(\Delta)\leq \frac{h}{2}$. As one has $$h=h(j_E)=h(\Delta)+\frac{1}{d'}\sum_{\substack{v\in{M_{K'}}\\v\vert\infty}} n_v\log\max\{\vert j_E\vert_v, 1\},$$ there exists an archimedean place $v$ such that $\log\max\{\vert j_E\vert_v, 1\}\geq \frac{h}{2}$, hence one deduces $h_v(E)\geq \frac{h}{2}$. So (\ref{pivot}) implies

\begin{equation}
\frac{\Reg(E/K)}{(\Card E(K)_{tors})^2}\geq c(d,m) \frac{ (h/2)^{m+2}}{ h^2} \geq c(d,m) h^m,
\end{equation}
which is a better lower bound than the one claimed.

Third case: assume $h(\Delta)\geq \frac{h}{2}$. Let $S'$ stand for the number of places in $K'$ where $E$ has multiplicative reduction. We have in particular $S'\neq0$ and by considering $v$ the finite place of $K'$ with maximal $h_v(E)$, one gets by the semi-stability assumption

\begin{equation}\label{hv}
h_v(E)\geq \frac{h(\Delta) d'}{S'}
\end{equation}
and by Lemma \ref{chebych} there is an absolute $c_0>0$ such that 

\begin{equation}
\frac{h(\Delta)d'}{\sigma}=\log N(F(E/K'))\geq c_0 S'\log(\frac{S'}{d'}+2),
\end{equation}
hence

\begin{equation}\label{sig}
\frac{h(\Delta)d'}{c_0 S'\log(\frac{S'}{d'}+2)}\geq  \sigma,
\end{equation}
inject (\ref{hv}) and (\ref{sig}) in (\ref{pivot}), and get:

\begin{equation}\label{piv}
\frac{\Reg(E/K)}{(\Card E(K)_{tors})^2}\geq c(d,m) \max\left\{\frac{ h(\Delta)^{m}}{(S')^{m+2} },  \frac{ h(\Delta)^{m}}{\Big(\frac{h(\Delta)}{S'\log(\frac{S'}{d'}+2)}\Big)^{2m+4} \Big(\log(3h(\Delta))\Big)^2}  \right\}.
\end{equation}

Subcase (i): assume $S'\leq d' h(\Delta)^{\frac{2}{3}} (\log(3h(\Delta))^{-\frac{2m+2}{3m+6}}$. Then the first part of (\ref{piv}) implies 

\begin{equation}
\frac{\Reg(E/K)}{(\Card E(K)_{tors})^2}\geq c(d,m) h(\Delta)^{m-\frac{2}{3}(m+2)}\Big(\log(3h(\Delta))\Big)^{\frac{(2m+2)(m+2)}{3m+6}},
\end{equation}
which gives directly

\begin{equation}
\frac{\Reg(E/K)}{(\Card E(K)_{tors})^2}\geq c(d,m)  h^{\frac{m-4}{3}}\Big(\log(3h)\Big)^{\frac{2m+2}{3}},
\end{equation}
as claimed.

Subcase (ii): assume $S'> d' h(\Delta)^{\frac{2}{3}} (\log(3h(\Delta))^{-\frac{2m+2}{3m+6}}$. Then the second part of (\ref{piv}) implies 

\begin{equation}
\frac{\Reg(E/K)}{(\Card E(K)_{tors})^2}\geq c(d,m) h(\Delta)^{m-\frac{1}{3}(2m+4)}\Big(\log(3h(\Delta))\Big)^{\frac{2m+2}{3}},
\end{equation} 
which ends the analysis of this last case with

\begin{equation}
\frac{\Reg(E/K)}{(\Card E(K)_{tors})^2}\geq c(d,m) h^{\frac{m-4}{3}}\Big(\log(3h)\Big)^{\frac{2m+2}{3}}.
\end{equation} 
A careful study of the different cases gives the claim for $c(d,m)$ and this concludes the whole proof of Theorem \ref{blichreg}.\\

To obtain Corollary \ref{NorthReg}, note that a bounded regulator and a fixed rank $m\geq4$ implies a bounded $j$-invariant, hence the claimed finiteness.
\end{proof}

\begin{rem}
There exists infinitely many elliptic curves over $\mathbb{Q}$ with rank at least $19$, we refer to the work of Elkies \cite{El08}. They are obtained by specialization of an elliptic curve with rank at least $19$ over $\mathbb{Q}(T)$. 
\end{rem}

\begin{rem}
It is easier to produce elliptic curves of rank at least $5$ over $\mathbb{Q}(T)$, and we thank Jean-Fran\c{c}ois Mestre for pointing this out. Pick a generic long Weierstrass equation $(W): y^2+a_1xy+a_3y=x^3+a_2x^2+a_4x+a_6.$ We have five parameters $(a_1,a_3,a_2,a_4,a_6)$. Choose the points $(0,0), (1,1), (2,3), (3,-1), (T,2)$. There exists a unique non-constant set of parameters $(a_1,a_3,a_2,a_4,a_6)\in{\mathbb{Q}(T)^5}$ such that $(W)$ passes through these five points. These points are generically independent on the curve defined by $(W)$. See also \cite{Mes} for explicit examples with rank bigger than 11 over $\mathbb{Q}(T)$. Then use specialization to obtain families over $\mathbb{Q}$ with big ranks.
\end{rem}

\section{Regulators of elliptic curves over function fields of characteristic zero}\label{ellFF0}

We will combine lemma \ref{success} with the following theorem.

\begin{theorem}\label{ff0}(Hindry-Silverman \cite{HiSi3}, Theorem 6.1 page 436)
Let $K$ be a function field of characteristic zero and genus $g$. Let $E/K$ be an elliptic curve of discriminant $\Delta(E/K)$ and let $P\in{E(K)}$ be a non-torsion point. Then $$\hat{h}_E(P)\geq 10^{-15.5-23g} h(E),$$ where $h(E)=\frac{1}{12}\deg \Delta(E/K)$. The same inequality is valid for elliptic curves over function fields in characteristic $p>0$ if the $j$-map has inseparable degree $1$.
\end{theorem}

We can now deduce the following statement by looking at Moret-Bailly's results on families of abelian varieties with bounded height. We first recall the definition of a \textit{limited family}\footnote{\textit{Confer} D\'efinition 1.1 page 212 of \cite{MB85} for limited families in more general settings.}.

\begin{definition}
Let  ${\mathcal A}_1$ be the coarse moduli space of elliptic curves over $K$ (\textit{i.e.} the affine line of $j$-invariants). Let ${\mathcal E}$ be a set of elliptic curves defined over $K$. The set ${\mathcal E}$ is limited if there exists a variety $T$ over the field of constants $k\subset K$, an open set $U\subset T_K$ surjective on $T$ and a $k$-morphism $f:U\to {\mathcal A}_1$ such that for any $E\in{\mathcal E}$, there is an $x\in T(k)$ with $f(x_K)$ being the image of $E$ in ${\mathcal A}_1(K)$.
\end{definition}

\begin{corollary}
Let $K$ be a function field of characteristic zero and genus $g$. The set of elliptic curves of trace zero over $K$, with non-trivial bounded rank and bounded regulator is limited.
\end{corollary}

\begin{proof}
Apply Theorem \ref{ff0} and (\ref{Minkowski2}) to get the inequality $$\mathrm{Reg}_L(E/K)\geq \frac{1}{m^m} \Big(10^{-15.5-23g} h(E)\Big)^m.$$ As soon as one fixes $m>0$, a bounded regulator implies a bounded height. Then apply \cite{MB85} Th\'eor\`eme 4.6 page 236.
\end{proof}

\section{Regulators of elliptic curves over function fields of positive characteristic}\label{ellFFp}

We will combine lemma \ref{success} with the following consequence of \cite{HiSi3}, which improves on Lemma 6.6 page 998 of \cite{Nas16}.

\begin{theorem}(based on Hindry-Silverman \cite{HiSi3}, see also \cite{Nas16})\label{ffp}
Let $P$ be a point of infinite order on $E$ over $k(\mathcal{C})$ of positive characteristic $p$ and genus $g$, and assume that the $j$-map of $E$ has inseparable degree $p^e$. Then one has $$\widehat{h}_E(P)\geq p^{-2e} 10^{-15.5-23g} h(E).$$
\end{theorem}

\begin{proof}
There exists a curve $E_0$ of inseparable degree $1$ such that $\phi^{e}:E_0\to E=E_0^{(p^{e})}$ is the $e$-th Frobenius isogeny. Denote the dual isogeny by $\widehat{\phi}^{e}$. We compute $$\widehat{h}_E(P)=p^{-e}\widehat{h}_{E_0}(\widehat{\phi}^{e}(P))\geq p^{-e}10^{-15.5-23g} h(E_0),$$
where this last inequality comes from the separable case (see Theorem \ref{ff0}) and since $h(E)=p^{e}h(E_0)$, we are done.
\end{proof}

In link with this Theorem \ref{ffp}, the interested reader should note that the remark page 434 of \cite{HiSi3} is inaccurate. Moreover, in Theorem 7 of \cite{GoSz} one should assume $e=0$.

\begin{corollary}
Let $K$ be a function field of characteristic $p>0$ and genus $g$. The set of elliptic curves of trace zero over $K$, with non-trivial bounded rank, bounded inseparability degree and bounded regulator is limited, and finite when $K$ has finite constant field.
\end{corollary}

\begin{proof}
Apply Theorem \ref{ffp} and (\ref{Minkowski2}) to get the inequality $$\mathrm{Reg}_L(E/K)\geq \frac{1}{m^m} \Big( p^{-2e} 10^{-15.5-23g} h(E)\Big)^m.$$ As soon as the rank $m$ is positive and bounded, and as long as $e$ is bounded, a bounded regulator implies a bounded height. Then apply \cite{MB85} Th\'eor\`eme 4.6 page 236.
\end{proof}

\section{Questions}\label{questions}

\subsection{Regulator and rank}

\begin{ques}
Let $E$ be an elliptic curve over a number field $K$. Let $m_K$ be its Mordell-Weil rank. Can one prove 
\begin{equation}\label{Q1}
\mathrm{Reg}(E/K)\geq c_0 \,m_K,
\end{equation}
where $c_0>0$ is a constant depending at most on $K$?
\end{ques}

The case of rank zero is trivial. Already in rank one, it would imply a non-trivial lower bound on the height of a generator of $E(K)$.
\\

\subsection{Small points on an elliptic curve}
A weaker version of Lang's conjectural inequality would already provide new insights. 
\begin{ques}
Let $E$ be an elliptic curve over a number field $K$. Let $P$ be a $K$-rational point of infinite order. Can one prove 
\begin{equation}\label{Q2}
\widehat{h}_E(P)\geq f(h(E)),
\end{equation}
where $f$ is a function tending to infinity when $h(E)$ tends to infinity?
\end{ques}
The interest here comes from the following remark: if one combines (\ref{Q2}) with inequality (\ref{Minkowski2}), this would imply a big improvement on Corollary \ref{NorthReg} by replacing the condition $m\geq4$ by $m\geq1$. Note that the inequality $\widehat{h}_E(P)\geq c_0>0$, where $c_0$ depends at most on $K$, is already an open and interesting problem though.

\subsection{Regulator and injectivity volume}

The Lang-Silverman conjecture in dimension $g\geq 1$, or the ABC conjecture in dimension 1, imply that on a simple principally polarized abelian variety $(A,L)$ of dimension $g$ over a number field $K$, when the rank is non-zero we have $\lambda_1^2\gg h(A)$. Let $\rho_\sigma(A,L)$ be the injectivity diameter of $A_\sigma(\mathbb{C})$ and $V_{\sigma}(A,L)$ be the injectivity volume, as described in \cite{Au16}. We would thus get by the Matrix Lemma $$m\RegL(A/K)^{1/m}\gg  \lambda_1^2 \gg h(A)\gg \sum_{\sigma:K\hookrightarrow\mathbb{C}}\frac{1}{\rho_\sigma(A,L)^2}\gg \sum_{\sigma:K\hookrightarrow\mathbb{C}}\frac{1}{V_\sigma(A,L)^{1/g}},$$ where all the implied constants are positive and may depend on $K$ and $g$.

\begin{ques}
Can one prove independently 
\begin{equation}\label{Q3}
\RegL(A/K)\gg \left(\displaystyle{\sum_{\sigma:K\hookrightarrow\mathbb{C}}}\frac{1}{V_\sigma(A,L)^{1/g}}\right)^{c_0 m},
\end{equation}
for a universal $c_0>0$? For elliptic curve with $m\geq 5$, the value $c_0=\frac{1}{15}$ is valid from Theorem \ref{blichreg}, with an implied constant depending on $K$ and on $m$.
\end{ques}



\begin{thebibliography}{widest-label}

\bibitem[Aut16]{Au16} \textsc{Autissier, P.}, 
\textit{Vari\'et\'es ab\'eliennes et th\'eor\`eme de Minkowski-Hlawka}.
Manusc. Math. {\bf149} (2016), 275--281.

\bibitem[BoGu06]{BG}  \textsc{Bombieri, E. and Gubler, W.}, 
\textit{Heights in Diophantine Geometry}.
New Mathematical Monographs {\bf 4}, Cambridge University Press (2006).

\bibitem[CoSi86]{CoSi} \textsc{Cornell, G. and Silverman, J.H.}, (editors) \textit{Arithmetic Geometry}, Springer-Verlag {\bf } (1986).

\bibitem[Con06]{Co06} \textsc{Conrad, B.}, 
\textit{Chow's $K/k$-image and $K/k$-trace, and the Lang-N\'eron theorem}.
Enseign. Math. (2) {\bf52} (2006), no. 1-2, 37--108.

\bibitem[Dav97]{Da97} \textsc{David, S.}, 
\textit{Points de petite hauteur sur les courbes elliptiques}.
J. Number Theory {\bf64} (1997), 104--129.

\bibitem[Elk08]{El08} \textsc{Elkies, N. D.}, 
\textit{Shimura Curve Computations Via K3 Surfaces of N\'eron-Severi Rank at Least 19}.
Lecture Notes in Computer Science 5011 (proceedings of ANTS-8, 2008; A.J.van der Poorten and A.Stein, eds.) (2008), 196--211.

 \bibitem[GoSz95]{GoSz}  \textsc{Goldfeld, D. and Szpiro, L.}, 
\textit{Bounds for the order of the Tate-Shafarevich group}.
Compositio Math. {\bf 97} (1995), 71--87.

 \bibitem[GrLe87]{GL}  \textsc{Gruber, P. and Lekkerkerker, C.G.}, 
\textit{Geometry of Numbers}.
 {\bf } North Holland Mathematical Library, Elsevier Science Publishers, (1987).

\bibitem[HiSi88]{HiSi3} \textsc{Hindry, M. and Silverman, J.H.}, 
\textit{The canonical height and integral points on elliptic curves}.
Invent. Math. {\bf93} (1988), 419--450.

\bibitem[Hin07]{Hin} \textsc{Hindry, M.}, 
\textit{Why is it difficult to compute the Mordell-Weil group?} Diophantine geometry, CRM Series, Ed. Norm., Pisa {\bf 4} (2007), 197--219.

\bibitem[HiPa16]{HiPa} \textsc{Hindry, M. and Pacheco, A.}, 
\textit{An analogue of the Brauer-Siegel theorem for abelian varieties in positive characteristic}, Mosc. Math. J.  {\bf 16} (2016), 45--93.

\bibitem[Lan83]{Lang} \textsc{Lang, S.}, 
\textit{Fundamentals of Diophantine geometry}, Springer-Verlag, New-York (1983).

\bibitem[Mes91]{Mes} \textsc{Mestre, J.-F.}, 
\textit{Courbes elliptiques de rang $\geq 11$ sur $\mathbb{Q}(t)$}.
C. R. Acad. Sc. Paris {\bf313} (1991), 139--142.

 \bibitem[MB85]{MB85}  \textsc{Moret-Bailly, L.}, 
\textit{Pinceaux de vari\'et\'es ab\'eliennes}.
Ast\'erisque {\bf 129}  (1985).

\bibitem[Nas16]{Nas16} \textsc{Naskrecki, B.}, 
\textit{Divisibility sequences of polynomials and heights estimates}.
New York J. Math. {\bf22} (2016), 989--1020.

\bibitem[Paz14]{Paz14} \textsc{Pazuki, F.}, 
\textit{Heights and regulators of number fields and elliptic curves}. Publ. Math. Besan\c{c}on {\bf 2014/2} (2014), 47--62. \textit{Erratum and addendum}. {\bf 2016} (2016), 81--83.

\bibitem[Paz16]{Paz16} \textsc{Pazuki, F.}, 
\textit{Northcott property for the regulators of number fields and abelian varieties}. Oberwolfach Rep. {\bf 21} (2016).

\bibitem[Pet06]{Pet06} \textsc{Petsche, C.}, 
\textit{Small points on elliptic curves over number fields}.
New York J. of Math. {\bf12} (2006), 257--268.

\end{thebibliography}
\end{document}